\documentclass[11pt]{amsart}
\usepackage{geometry}
\usepackage{graphicx}
\usepackage{amssymb}
\usepackage{amsmath}
\usepackage{color}
\newtheorem{theorem}{Theorem}

\newtheorem{lemma}[theorem]{Lemma}
\newtheorem{corollary}[theorem]{Corollary}
\newtheorem{proposition}[theorem]{Proposition}
\newtheorem{remark}[theorem]{Remark}

\newcommand{\R}{\mathbb R}
\numberwithin{theorem}{section}
\numberwithin{equation}{section}
\begin{document}
\title{Quasi-local mass at the null infinity of the Vaidya spacetime}
\author{Po-Ning Chen, Mu-Tao Wang, and Shing-Tung Yau}
\date{\today}
\begin{abstract}There are two important statements regarding the Trautman-Bondi mass \cite{Bondi-Burg-Metzner, Trautman, CJM} at null infinity: one is the positivity \cite{Schoen-Yau, Horowitz-Perry}, and the other 
is the Bondi mass loss formula \cite{Bondi-Burg-Metzner}, which are both global in nature. The positivity of the quasi-local mass can potentially lead to a local description
at null infinity. This is confirmed for the Vaidya spacetime in this note. We study the Wang-Yau quasi-local mass on surfaces of fixed size at the null infinity of the Vaidya spacetime. The optimal embedding equation is solved explicitly and the quasi-local mass is evaluated in terms of the mass aspect function of the Vaidya spacetime. 
\end{abstract}
\thanks{P.-N. Chen is supported by NSF grant DMS-1308164, M.-T. Wang is supported by NSF grants DMS-1105483 and DMS-1405152.  and S.-T. Yau is supported by NSF
grants  PHY-0714648 and DMS-1308244.  Part of this work was carried out
when P.-N. Chen and M.-T. Wang were visiting the Department of Mathematics and the Center of Mathematical Sciences and Applications at Harvard University and the National Center of 
Theoretical Sciences at National Taiwan University.} 
\maketitle
\section{Introduction}
In \cite{Chen-Wang-Yau6,Chen-Wang-Yau5}, we study gravitational perturbations of black holes by evaluating the Wang-Yau quasi-local mass on surfaces of fixed size near the infinity of the asymptotically flat spacetime. In particular, we prove a theorem regarding the decay rate of the quasi-local energy-momentum at the infinity of gravitational perturbations of the Schwarzschild spacetime. Namely, in Theorem 5.4 of \cite{Chen-Wang-Yau5},  we proved that the quasi-local mass on spheres of unit size decays at the rate of $O(d^{-2})$ for gravitational perturbations of the Schwarzschild spacetime. We also evaluated the quasi-local mass for two explicit examples, the polar perturbations and axial perturbations of the Schwarzschild spacetime introduced by Chandrasekhar \cite{Chandrasekhar}.  These are perturbations of the Schwarzschild spacetime solving the linearized vacuum Einstein equation. 

In this article, we follow the ideas and techniques developed in \cite{Chen-Wang-Yau5} to evaluate the Wang-Yau quasi-local mass \cite{Wang-Yau1,Wang-Yau2} on spheres of unit size at the null infinity of the Vaidya spacetime. The Vaidya spacetime describes the exterior region of a spherically symmetric star emitting or absorbing null dust. While it is not a small perturbation of the Schwarzschild spacetime, we show that the quasi-local mass decays at the same rate as predicted by Theorem 5.4 of \cite{Chen-Wang-Yau5}. In particular, we can apply Theorem 5.2 of \cite{Chen-Wang-Yau5} (see Theorem \ref{mass_limit_general} below) to evaluate the quasi-local mass in terms of the expansion of the induced metric and mean curvature of the surface, as well as the solution of the optimal embedding equation. 

To evaluate the quasi-local mass, we first solved the optimal embedding equation explicitly in terms of the mass aspect function of the Vaidya spacetime (see Lemma \ref{sol_optimal_first_order} below). This solution is used to evaluate the quasi-local mass in terms of the mass aspect function (see Theorem \ref{main_theorem} below). In particular, from Theorem \ref{main_theorem}, we observe directly that the quasi-local mass is positive if the mass aspect function is monotonically decreasing.

The article is organized as follows: in Section \ref{Vaidya}, we review the construction, notation and results from \cite{Chen-Wang-Yau5} that we will use in this article. We also give a brief review of the Vaidya spacetime. In Section \ref{section_optimal}, we solve the optimal embedding equation and evaluate the terms in Theorem 5.2 of \cite{Chen-Wang-Yau5}  that depend on the equation. In Section \ref{induced_data}, we compute the expansion of the data to the order of $O(d^{-2})$ and finish the evaluation of the quasi-local mass.
\section{Unit size spheres at null infinity of the Vaidya spacetime} \label{Vaidya}
The Vaidya metric, in the $(u, r, \theta, \phi)$ coordinate system, is of the form
\begin{equation}
-(1-\frac{M(u)}{r}) du^2 -2dudr + r^2 (d\theta^2+\sin^2\theta d\phi^2)
\end{equation}
The stress-energy tensor of the Vaidya metric is 
\[ T_{\alpha\beta}= -\frac{\partial_u M(u)}{4 \pi r^2}\nabla_\alpha u  \nabla_\beta u \] with respect to any spacetime coordinate system $x^\alpha, \alpha=0, 1, 2, 3$. 
In particular, the dominant energy condition is satisfied if $M$ is monotonically non-increasing. The function $M$ is referred to as the mass aspect function.

We consider a null geodesic of the form $u=u_0, \theta=\theta_0, \phi=\phi_0$ where $u_0, \theta_0, \phi_0$ are all constants, and ``unit spheres" 
centered at points on such a null geodesic. The limit of the quasi-local mass of these spheres is then evaluated at null infinity. 

Consider a new coordinate system $(t, y^1, y^2, y^3)$ with $t=u+r, y_1= r\sin\theta\sin \phi, y_2= r\sin\theta \cos \phi,$ and $y_3= r\cos \theta$.
A null geodesic with $u=0, \theta=\theta_0, \phi=\phi_0$ corresponds to points with the new coordinates $(d, d_1, d_2, d_3)$, where $d  = \sqrt{\sum_i d_i^2} $.

With fixed $\tilde{d}_i=\frac{d_i}{d}, i=1, 2, 3$, we consider the sphere $\Sigma_d$ of (Euclidean) radius $1$ centered at a point $ (d, d_1, d_2, d_3) $ on the null geodesic.  Namely,
\begin{equation}\label{unit_sphere}
\Sigma_{d}= \{(t, y^1, y^2, y^3)| \,\,t=d, \sum_i (y^i-d_i)^2 =1\}.
 \end{equation}

In the following, we study the family of surfaces $\Sigma_d$ defined in \eqref{unit_sphere} as $d\rightarrow \infty$.
To study the geometry of these surfaces, we consider another spherical coordinate system $\{s,  u^{a} \}_{a=1, 2}$ centered at $(d_1,d_2,d_3)$ on the $t=d$ slice. The coordinate transformation between $\{s, u^1, u^2\}$ and $\{y^1, y^2, y^3\}$ is given by
\begin{equation}\label{coord_transf} d_i + s \tilde X^ i (u^1, u^2)= {y}^i, i=1, 2, 3 \end{equation} where $\tilde X^i (u^a), i=1, 2, 3$ are the three standard coordinate functions on a unit sphere. In particular, from \eqref{coord_transf} we derive: 
\begin{equation}\label{r_and_d} r^2= d^2+ 2s \sum_i d_i \tilde X^i + s^2.\end{equation}
In the new coordinate system $(t, s, u^1, u^2)$, $\Sigma_{d}$ is simply the sphere with $t=d$ and $s=1$.  We define the function \begin{equation}\label{Z} Z(u^a)= \tilde d_i \tilde X^i (u^a)\end{equation} on $S^2$. (This is denoted by $Z_1$ in \cite{Chen-Wang-Yau5}).  From \eqref{coord_transf}, we solve
\[\begin{split} dr&=(\sum_i \tilde{X}^i \tilde y^i) ds+(s\sum\tilde{X}^i_b \tilde{y}^i) du^b\\
dv^A&=(\frac{1}{r}\sum_i\tilde{X}^i \tilde{\nabla}^A \tilde{y}^i) ds+(\frac{s}{r}\sum_i\tilde{X}^i_b\tilde{\nabla}^A\tilde{y}^i) du^b, \end{split}\] where $A=1, 2$, $v^1=\theta, v^2=\phi$, and $\tilde{y}^i=\frac{y^i}{r}, i=1, 2, 3$.

In terms of the coordinate system $\{s, u^1, u^2\}$, the metric on the hypersurface $t=d$ is
\[  \bar g_{ss} ds^2 + 2 \bar g_{s a} ds du^{ a} + \bar g_{ab} du^{a}  du^{b} \]
where
\begin{equation} \label{new_metric}
\begin{split}
\bar g_{ss}= & 1+ \frac{M(u)}{r} \sum_{i,j} (\tilde X^i \tilde y^i \tilde X^j \tilde y^j) \\
\bar g_{sa}=& \frac{sM(u)}{r}  \sum_{i,j} (\tilde X^i \tilde y^i \tilde X^j_{a} \tilde{y}^j )\\
\bar g_{{a}{b}}= &s^2 [\tilde \sigma_{ab}+   \frac{M(u)}{r} \sum_{i,j}  (\tilde X^i_{a} \tilde{y}^i \tilde X^j_{b} \tilde{y}^j)].
\end{split}
\end{equation}

From \eqref{coord_transf} and \eqref{r_and_d} we derive the following expansions on  $\Sigma_d$:
\begin{equation} \label{useful_expansion}
\begin{split}
r=& d+ Z + \frac{1-Z^2}{2d}+ O(d^{-2})\\
\sum_i \tilde X^i \tilde y^i = & Z + \frac{1-Z^2}{d}+ O(d^{-2})\\
\sum_i \tilde X_a^i \tilde y^i = & Z_a - \frac{ZZ_a}{d}+ O(d^{-2})\\
\end{split}
\end{equation}
and
\begin{equation}  \label{useful_expansion2}
\begin{split}
\frac{\partial r}{\partial s} = & Z+ \frac{1-Z^2}{d}+ O(d^{-2})\\
\sum_i \tilde X^i \frac{ \partial \tilde y^i}{\partial s} = &  \frac{1-Z^2}{d}+ O(d^{-2})\\
\sum_i  \tilde X_a^i  \frac{ \partial \tilde y^i}{\partial s}= &  - \frac{ZZ_a}{d}+ O(d^{-2}).\\
\end{split}
\end{equation}
The restriction of $u$ to $\Sigma_{d}$ is a function on $S^2$ which also depends on $d$. As a result, the restriction of the mass aspect function $M(u)$ to $\Sigma_{d}$ is a function $M(d,u^a)$.
To evaluate the limit of the quasi-local mass, we consider $\lim_{d \to \infty} M(d,u^a)$. From \eqref{useful_expansion}, we conclude that on $\Sigma_{d}$, we have
\begin{equation}\label{u_expansion}u=-Z-\frac{1-Z^2}{2d}+ O(d^{-2})  \end{equation}
and there is a function $F$ of a single variable defined on  $[-1,1]$ such that
\begin{equation} \label{mass_aspect_limit}\lim_{d \to \infty} M= F(Z) \end{equation}
where $F(Z)$ is the composition of $F$ with $Z(u^a)$.
From \eqref{mass_aspect_limit}, it is easy to see that 
\[ M(u)= F(Z) + O(d^{-1}).\]
Similarly, we have
\[
\begin{split}
\partial_u M(u)&= -F'(Z) + O(d^{-1}) \\
 \partial^2_u M(u)&=  F''(Z) + O(d^{-1}) \\
  \partial^3_u M(u)&= -F'''(Z) + O(d^{-1}),
\end{split}
\]
where $F'$ denotes the derivative of $F$ .

Let $\sigma_{ab}$, $|H|$ and $\alpha_H$ be the induced metric, the norm of the mean curvature vector, and the connection 1-form in mean curvature gauge of the surface $\Sigma_{d}$, respectively. It is easy to see that they admit the following expansions:
\begin{equation} \label{expansion_general}
\begin{split}
\sigma_{ab} = & \tilde \sigma_{ab} + \frac{\sigma^{(-1)}_{ab}}{d} + \frac{\sigma^{(-2)}_{ab}}{d^2} + O(d^{-3})\\
|H| = &2 + \frac{h^{(-1)}}{d}+ \frac{h^{(-2)}}{d^2} + O(d^{-3})\\
(\alpha_H)_a = &\frac{(\alpha_H^{(-1)})_a}{d} + \frac{(\alpha_H^{(-2)})_a}{d^2}+ O(d^{-3}).
\end{split}
\end{equation}
In particular, we have the following expansion for the area form
\begin{equation}\label{expansion_area} {\rm d}V_\sigma = (1+ \frac{V^{(-1)}}{d} + \frac{V^{(-2)}}{d^2}  ) {\rm d}V_{S^2} +  O(d^{-3}). \end{equation}

The quasi-local mass for surfaces admitting such data is evaluated in \cite[Theorem 5.2]{Chen-Wang-Yau5}. 
\begin{theorem} \label{mass_limit_general}
Let $\Sigma_d$ be  a family of surfaces whose data $(\sigma,|H|,\alpha_H)$ admits the expansions given by \eqref{expansion_general}.  Assume further that 
\begin{align}
\label{second_order_cond_1}\int_{S^2} (h_0^{(-1)} -h^{(-1)}) \, {\rm d}V_{S^2}=& 0\\
\label{second_order_cond_2}\int_{S^2} \tilde X^i \tilde \nabla^a (\alpha_H^{(-1)})_a  \, {\rm d}V_{S^2}=&0.  \end{align}
Let $X$ be the isometric embedding of $\Sigma_d$ into $\R^{3,1}$ such that  $X^0=O(d^{-1})$ and the pair $X$  and $T_0=(1,0,0,0)$ solves the leading order
term of the optimal embedding equation. Then
\begin{equation}\label{mass_general}
\begin{split}
  & E(\Sigma_d ,X,T_0)\\
=& \frac{1}{8 \pi d^2}\int_{S^2} \Big \{ \frac{|\tilde\nabla N |^2}{2} -N^2  - \frac{1}{4}(X^0)^{(-1)} \widetilde  \Delta (\widetilde  \Delta + 2 ) (X^0)^{(-1)} \\
   & \qquad \qquad  \,\,\,\, -V^{(-2)} - h^{(-2)} -h^{(-1)}V^{(-1)}   \Big \} \, {\rm d}V_{S^2} +O(d^{-3})
\end{split}
 \end{equation}
where $(X^0)^{(-1)}$ and $N$ satisfy
\begin{equation}\label{optimal_first_order} 
\begin{split}
\widetilde \Delta (\widetilde  \Delta + 2 ) (X^0)^{(-1)}= & 2 \tilde \nabla^a (\alpha_H^{(-1)})_a\\
  -(\widetilde  \Delta + 2 )N =& \frac{1}{2}\left [  \tilde\nabla^a \tilde\nabla^b \sigma^{(-1)}_{ab} - tr_{S^2} \sigma^{(-1)} -\widetilde \Delta( tr_{S^2} \sigma^{(-1)}) \right].
\end{split}
\end{equation}  
\end{theorem}
\begin{remark}
In fact, \eqref{optimal_first_order}  is the leading order term of the optimal embedding equation. While the solution to \eqref{optimal_first_order}  is not unique, the leading order term of the quasi-local energy is independent of the choice of the solution. 
\end{remark}
We evaluate the quasi-local mass for $\Sigma_{d}$ in the next two sections. In Section \ref{section_optimal}, we solve $(X^0)^{(-1)}$ and $N$ from \eqref{optimal_first_order}  and evaluate the terms in \eqref{mass_general} involving $(X^0)^{(-1)}$ and $N$. In Section \ref{induced_data}, we evaluate the terms in \eqref{mass_general} involving $h^{(-1)}$, $h^{(-2)}$, $V^{(-1)} $ and $V^{(-2)} $.

\section{Optimal embedding equation} \label{section_optimal}
We first compute $ \alpha_H^{(-1)}$ and $ \sigma^{(-1)}_{ab} $. 
\begin{lemma}
For the data $ \alpha_H^{(-1)}$ and $ \sigma^{(-1)}_{ab} $, we have
\begin{equation} \label{lemma 3.1}
\begin{split}
\sigma^{(-1)}_{ab} = & F Z_a Z_b\\
(\alpha_H^{(-1)})_a= &  -F' Z Z_a + \frac{1}{4} F''(1-Z^2)Z_a.
\end{split}
\end{equation}
\end{lemma}
\begin{proof}
From \eqref{new_metric}, \eqref{useful_expansion} and \eqref{u_expansion}, we immediately obtain the expression for $\sigma^{(-1)}_{ab} $. Next we  compute the second fundamental form of the hypersurface. The unit normal vector, up to lower order terms, is 
\[ \frac{\partial }{\partial t} + \frac{M(u)}{r} \frac{\partial }{\partial r}   \]
and all the components of the second fundamental form of $t=d$ vanish except 
\[  k_{rr} =\frac{F'}{2d} + O(d^{-2}).\]
On $\Sigma_{d}$, we have 
\[  dr =  Z ds + Z_a d  u^a +O(d^{-1})\]
and 
\[
\begin{split}
k_{sa} =&   \frac{F'}{2d} Z Z_a + O(d^{-2}) \\
k_{ab}= & \frac{F'}{2d} Z_a Z_b + O(d^{-2}).
\end{split}
\]
The lemma follows from \[ (\alpha_H)_a = -k(\nu, \partial_a)+ \nabla_a \frac{tr_{\Sigma}k }{|H|}. \] 
\end{proof}
We compute the right hand side of \eqref{optimal_first_order} in the following lemma.
\begin{lemma}
\begin{equation}\label{right_hand_side_1}
\frac{1}{2} \left [\tilde\nabla^a \tilde\nabla^b \sigma^{(-1)}_{ab} - tr_{S^2} \sigma^{(-1)} -\widetilde \Delta( tr_{S^2} \sigma^{(-1)})  \right] =-\frac{1}{2}F' Z (1-Z^2) - F(1-2Z^2) .
\end{equation}
and 
\begin{equation}\label{right_hand_side_2}
2 \tilde \nabla^a (\alpha_H^{(-1)})_a  = \frac{1}{2} F''' (1-Z^2) ^2 - 4 F'' Z(1-Z^2) - 2 F' (1-3Z^2)
\end{equation}
\end{lemma}
\begin{proof}
Recall
\[  \sigma^{(-1)}_{ab} = FZ_aZ_b  \] 
and thus
\[  tr_{S^2}\sigma^{(-1)}_{ab} = F(1-Z^2).  \] 
We compute
\[
\begin{split}
\widetilde \Delta( tr_{S^2} \sigma^{(-1)})  =& (\widetilde \Delta F)(1-Z^2)   -4Z \tilde \nabla _a F  \tilde \nabla^a Z + F \widetilde \Delta (1-Z^2 )\\
= &F'' (1-Z^2)^2 - 6 F' Z(1-Z^2)^2+ F(6Z^2-2)
\end{split}
\]
and 
\[  
\begin{split}
\tilde\nabla^a \tilde\nabla^b \sigma^{(-1)}_{ab} = & \tilde \nabla^a [ F' (1-Z^2) Z_a-3FZZ_a]\\
=& F'' (1-Z^2)^2  - 7 F' Z(1-Z^2)^2 -3F(1-3Z^2)
\end{split}
\]
Collecting terms yields \eqref{right_hand_side_1}. For  \eqref{right_hand_side_2}, we compute
\[
\begin{split}
\tilde \nabla^a (\alpha_H^{(-1)})_a= & \tilde \nabla^a \left [ -F' Z Z_a + \frac{1}{4}\partial_u^2 F(1-Z^2)Z_a \right]\\
=& \frac{1}{4} F'''(1-Z^2)^2 -2 F'' Z(1-Z^2)-F' (1-3Z^2).
\end{split}
\]
\end{proof}
\begin{lemma} Let $G$ be an anti-derivative of $F$ such that $G'=F$ and $G(Z)$ denote the composition of $G$ and $Z(u^a)$, then
\begin{equation} \label{sol_optimal_first_order}
\begin{split}
N = & \frac{1}{2} Z G(Z) \\
(X^0)^{(-1)} =  &\frac{1}{2} G(Z)
\end{split}
\end{equation}
solves the optimal embedding equation  \eqref{optimal_first_order}.
\end{lemma}
\begin{proof}
For axially symmetric functions depending on $Z$, we have 
\[  \widetilde  \Delta = \frac{\partial }{\partial Z} [ (1-Z^2) \frac{\partial }{\partial Z} ]. \]
Hence,
\[
\begin{split}
 -\widetilde  \Delta (Z  G(Z) )= & -\frac{\partial }{\partial Z} \left [ (1-Z^2)Z F+ (1-Z^2) G(Z) \right ] \\
=& -F' Z(1-Z^2) -F(1-3Z^2) - F(1-Z^2) + 2Z  G(Z)
\end{split}
\]
and 
\[
-\frac{1}{2}(\widetilde \Delta +2)  (Z  G(Z)  )=- \frac{1}{2} F'Z(1-Z^2) -F(1-2Z^2).
 \]
 
 Next we compute
 \[
 \begin{split}
  &  \frac{1}{2} \widetilde \Delta  ( \widetilde \Delta+2)(G(Z))\\
=&  \frac{1}{2}   ( \widetilde \Delta+2)  \frac{\partial }{\partial Z} (F(1-Z^2))\\
=& \frac{1}{2}   ( \widetilde \Delta+2) [F'(1-Z^2) -2Z F ]\\
=&  F'(1-Z^2) -2Z F + \frac{1}{2}  \frac{\partial }{\partial Z} [ (1-Z^2) \frac{\partial }{\partial Z} (F'(1-Z^2) -2Z F)]  \\
=& F'(1-Z^2) -2Z F + \frac{1}{2}  \frac{\partial }{\partial Z} [ F''(1-Z^2)^2 - 4F'Z(1-Z^2) - 2F(1-Z^2)] \\
=& \frac{1}{2} F''' (1-Z^2)^2  -4 F''Z(1-Z^2)+ 2F'(3Z^2-1).
\end{split}
 \]
\end{proof}
 \begin{corollary}
The condition \eqref{second_order_cond_1} holds for the family of surfaces $\Sigma_{d}$.
 \end{corollary}
 \begin{proof}
 This follows from the solvability of \eqref{optimal_first_order}.
 \end{proof}

Finally, we evaluate the integrals.  
\begin{lemma} \label{isometric_integral}
For the $N$ and $(X^0)^{(-1)}$ solving \eqref{optimal_first_order}, 
\begin{align}
\label{integral_N}\int_{S^2} \left[ \frac{|\tilde \nabla N|^2}{2} - N^2  \right]{\rm d}V_{S^2} = &  \int_{S^2} \frac{1}{8} F^2Z^2(1-Z^2) {\rm d}V_{S^2} \\
\label{integral_X}\int_{S^2} (X^0)^{(-1)}\widetilde  \Delta (\widetilde  \Delta+2) (X^0)^{(-1)} {\rm d}V_{S^2} =& \int_{S^2} \frac{1}{4} (F')^2(1-Z^2)^2 {\rm d}V_{S^2}
\end{align}
\end{lemma}
\begin{proof}
For any axially symmetric function $f(Z)$, we have
\[ |\tilde \nabla f|^2 = (1-Z^2) (f')^2.  \]
We compute 
\[  
\begin{split}
\frac{|\tilde \nabla N|^2}{2} - N^2 
=&   \frac{1}{8} (1-Z^2) ( G(Z) + ZF)^2 -\frac{1}{4} Z^2 (G(Z))^2\\
=&  \frac{1}{8} (1-Z^2) Z^2 F^2  +\frac{1}{8}(1-3Z^2) (G(Z))^2 + \frac{1}{4} (1-Z^2) ZF (G(Z)  )\\
=&   \frac{1}{8} (1-Z^2) Z^2 F^2  + \frac{1}{8} \frac{\partial}{\partial Z} [(G(Z))^2 (1-Z^2) Z].\\
\end{split}
\]
\eqref{integral_N} follows since for any function $f$ of a single variable, we have
\begin{equation} \label{vanish_by_part}\int_{S^2} f'(Z) {\rm d}V_{S^2}  = 2 \pi (f(-1)- f(1)). \end{equation}

On the other hand,
\[
\int_{S^2}  (X^0)^{(-1)}\widetilde  \Delta (\widetilde  \Delta+2) (X^0)^{(-1)}  {\rm d}V_{S^2}  = \int_{S^2}[ \widetilde  \Delta (X^0)^{(-1)}]^2 - 2 |\tilde \nabla (X^0)^{(-1)}|^2  {\rm d}V_{S^2} .
\]
We compute
\[
\begin{split}
  [ \widetilde  \Delta (X^0)^{(-1)}]^2 - 2 |\tilde \nabla (X^0)^{(-1)}|^2 
=&[\frac{1}{2} F'(1-Z^2) -ZF]^2 - \frac{1}{2} (1-Z^2) F^2  \\
= &  \frac{1}{4} (F')^2 (1-Z^2)  - FF'Z(1-Z^2) - \frac{1}{2} (1-3Z^2) F^2\\
= &  \frac{1}{4} (F')^2 (1-Z^2)  - \frac{1}{2} \frac{\partial}{\partial Z} [F^2Z(1-Z^2)].\\
\end{split}
\]
This implies \eqref{integral_X} .
\end{proof}
Combining Theorem \ref{mass_limit_general} and Lemma \ref{isometric_integral}, the  quasi-local mass of the surface $\Sigma_{d}$ is 
\[
\frac{1}{8 \pi d^2 }\int_{S^2} \left [ \frac{1}{8} (1-Z^2) Z^2 F^2 -  \frac{1}{16} (F')^2 (1-Z^2)  - h^{(-2)} - h^{(-1)}V^{(-1)} -V^{(-2)} \right] {\rm d}V_{S^2}+O(d^{-3})
\]
Let $\hat h$ be the mean curvature of $\Sigma_{d}$  in the hypersurface $t=d$. 
\[ |H|^2 = \sqrt{\hat h^2  - \frac{1}{4} (F')^2(1-Z^2)^2} +O(d^{-3}) . \]
Hence the  quasi-local energy of the surface $\Sigma_{d}$ is 
\begin{equation} \label{optimal_energy_final}
\frac{1}{8 \pi d^2 }\int_{S^2} \left [ \frac{1}{8} (1-Z^2) Z^2 F^2 - \hat h^{(-2)} - \hat h^{(-1)}V^{(-1)} -V^{(-2)} \right] {\rm d}V_{S^2} +O(d^{-3})
\end{equation}
where
\[  \hat h = 2 + \frac{ \hat h^{(-1)}}{d}+\frac{ \hat h^{(-2)}}{d^2} +O(d^{-3}).\]
\section{The expansion of mean curvature and area form} \label{induced_data}
In this section, we compute the expansions of the area form and the mean curvature of $\Sigma_{d}$. We start with the area form.
\begin{lemma} 
\begin{equation}\label{lemma 4.1}
\begin{split}
V^{(-1)} =& \frac{F(1-Z^2)}{2}\\
V^{(-2)} = &-\frac{3}{2}  F Z(1-Z^2) + \frac{1}{4} F' (1-Z^2)^2 - \frac{1}{8}
 F^2 (1-Z^2)^2.\end{split}
\end{equation}
\end{lemma}
\begin{proof}
We compute
\[  \frac{det(\sigma)}{det(\tilde \sigma)} =  1+ \frac{M(u)}{r} \sum_{ij}\tilde \sigma^{ab} \tilde X^i_a \tilde y^i  X^j_b \tilde y^j,\]
since 
\[  \sigma_{ab} = \tilde \sigma_{ab} + \frac{M(u)}{r}\sigma^{ab} \sum_{ij}\tilde X^i_a \tilde y^i  X^j_b \tilde y^j  \]
and 
\[ det (\sigma^{(-1)}) = 0 . \]
It follows that
\[  \frac{det(\sigma)}{det(\tilde \sigma)} =  1+ \frac{M(u)}{r}(1- \sum_{ij}\tilde X^i \tilde y^i  X^j \tilde y^j). \]
Using the expansion of $r$ and $\tilde X^i \tilde y^i $ in \eqref{useful_expansion} and 
\begin{equation} \label{expansion_M_0}M(u) =  F + \frac{1}{2d} F'(1-Z^2) + O(d^{-2}), \end{equation}
we conclude that 
\[  \frac{det(\sigma)}{det(\tilde \sigma)} =  1+ \frac{F(1-Z^2)}{d} + \frac{1}{d^2} \left [  -3FZ(1-Z^2) + \frac{1}{2} F'(1-Z^2) \right] + O(d^{-3}) .\]
The lemma follows from taking the square root of the above equation.
\end{proof}
Next we compute the expansion of the mean curvature $\hat h$. In terms of the induced metric $\bar g$ of the hypersurface $t=d$, 
\[ \hat h =   \frac{\frac{1}{2} \sigma^{ab} \partial_s \bar g_{ab} - \nabla ^a  \bar g_{as} }{\sqrt{ \bar g_{ss} - \sigma^{ab} \bar g_{as} \bar g_{bs}}}. \]
To make use of the divergence structure, we rewrite $\hat h$ as 
\[ \hat h =   \frac{\frac{1}{2} \sigma^{ab} \partial_s  \bar g_{ab}}{\sqrt{g_{ss} - \sigma^{ab} \bar g_{as} \bar g_{bs}}}   - (\nabla ^a  \bar g_{as})  \left (\frac{1}{\sqrt{ \bar g_{ss} - \sigma^{ab} \bar g_{as} \bar g_{bs}}}-1\right )  - \nabla ^a  \bar g_{as}  . \]
We do not need to compute the last term since its integral vanishes.  Let
\[ \frac{\frac{1}{2} \sigma^{ab} \partial_s  \bar g_{ab}}{\sqrt{ \bar g_{ss} - \sigma^{ab} \bar g_{as} \bar g_{bs}}}   - (\nabla ^a  \bar g_{as})  \left (\frac{1}{\sqrt{ \bar g_{ss} - \sigma^{ab} \bar g_{as} \bar g_{bs}}}-1\right) =  2+ \frac{\breve h^{(-1)} }{d} + + \frac{\breve h^{(-2)} }{d^2} +O(d^{-3}). \]
\begin{lemma}
\begin{equation} \label{lemma 4.2}
\begin{split}
\breve h^{(-1)} = &  -FZ^2 + \frac{1}{2}F' Z(1-Z^2) \\
\breve h^{(-2)} = &  \frac{1}{4} F'' Z(1-Z^2)^2 - \frac{F'}{2} (-1 + 6Z^2 - 5Z^4) + \frac{F}{2} (9Z^3 - 7Z)\\
& + \frac{F^2}{4} (6 Z^2 - 7Z^4) -FF'  \Big [ \frac{1}{2} Z(1-Z^2)^2-\frac{1}{4 } Z^3(1-Z^2) \Big ].
\end{split}
\end{equation}
\end{lemma}
\begin{proof}
Recall that 
\[
\begin{split}
  \bar g_{ss} = &1 + \frac{M(u)}{r} \sum_{i,j}\tilde X^i \tilde y^i \tilde X^j \tilde y^j\\
  \bar g_{as}  =& \frac{F Z Z_a}{d} + O(d^{-2}).
\end{split} 
 \] 
We compute 
\[ 
\begin{split}
   \frac{1}{\sqrt{ \bar g_{ss} - \sigma^{ab}  \bar g_{as} \bar g_{bs}}}
= & \frac{1}{\sqrt{1+ \frac{M(u)}{r}  \sum_{i,j} \tilde X^i \tilde y^i \tilde X^j \tilde y^j  -\frac{F^2 }{d^2} Z^2(1-Z^2)}} +O(d^{-3}) \\
=& \frac{1}{1 + \frac{M(u)}{2r} \sum_{i,j}\tilde X^i \tilde y^i \tilde X^j \tilde y^j - \frac{F^2}{d^2}  [ \frac{1}{2} Z^2(1-Z^2) +\frac{1}{8} Z^4]}+O(d^{-3}) \\
= & 1-  \frac{M(u)}{2r} \sum_{i,j}\tilde X^i \tilde y^i \tilde X^j \tilde y^j  + \frac{F^2}{d^2}  [ \frac{1}{2} Z^2(1-Z^2) +\frac{3}{8} Z^4]+O(d^{-3}) .
\end{split}
 \]
From the expansion of $M(u)$ in \eqref{expansion_M_0} and the expansions of $r$ and $\tilde X^i \tilde y^i$ in \eqref{useful_expansion}, we conclude
 \begin{equation} \label{mean_curvature_term_1}
 \begin{split}
  &\frac{1}{\sqrt{ \bar g_{ss} - \sigma^{ab}  \bar g_{as} \bar g_{bs}}}\\
  = & 1 - \frac{F Z^2}{2d} + \frac{1}{d^2} \left[ F^2 (\frac{Z^2}{2} - \frac{Z^4}{8}) - \frac{1}{4} F' Z^2(1-Z^2) + \frac{1}{2} F (3Z^3-2Z) \right]+ O(d^{-3}).
  \end{split}
 \end{equation}
Next we compute
\begin{equation} \label{divergence_leading} \nabla ^a  \bar g_{as}  = \frac{1}{d}  \left [F (1-3Z^2)  +F' Z(1-Z^2) \right ] + O(d^{-2}).   \end{equation}
As a result,
 \begin{equation} \label{mean_curvature_term_2}
  - (\nabla ^a  \bar g_{as} )\left (\frac{1}{\sqrt{ \bar g_{ss} - \sigma^{ab} \bar g_{as} \bar g_{bs}}}-1 \right)  = \frac{1}{2d^2} \left [ F^2(1-3Z^2) +FF'  Z^3(1-Z^2)   \right] + O(d^{-3}).
 \end{equation}
Finally, we compute $\sigma^{ab} \partial_s \bar g_{ab}$.
\[ \partial_s  \bar g_{ab}  = 2 \sigma_{ab} + \sum_{i,j} \frac{\partial}{\partial s} [ \frac{M(u)}{r} \tilde X^i_a \tilde y^i X^j_b \tilde y^j]  \]
and 
\[
\begin{split}
   \sigma^{ab} \partial_s  \bar  g_{ab} 
=&  4 +  \sum_{i,j} \sigma^{ab} \frac{\partial}{\partial s} [ \frac{M(u)}{r} \tilde X^i_a \tilde y^i X^j_b \tilde y^j]  \\
= & 4 + (\tilde \sigma^{ab} - \frac{F}{d} \tilde \nabla ^a Z \tilde \nabla^b Z )  \sum_{i,j}  \frac{\partial}{\partial s} [ \frac{M(u)}{r} \tilde X^i_a \tilde y^i X^j_b \tilde y^j] + O(d^{-3}) \\
=&  4 + (\tilde \sigma^{ab} - \frac{F}{d} \tilde \nabla ^a Z \tilde \nabla^b Z )    \frac{\partial_s M(u)}{r}   \sum_{i,j} \tilde X^i_a \tilde y^i X^j_b \tilde y^j    +  F \frac{\partial}{\partial s} [ \frac{\tilde \sigma^{ab}}{r}  \sum_{i,j} \tilde X^i_a \tilde y^i X^j_b \tilde y^j] +O(d^{-3})
\\
=& 4 +  \frac{\partial_s M(u)}{r} (1-  \sum_{i,j}\tilde X^i \tilde y^i  \tilde X^j \tilde y^j) -\frac{FF' Z(1-Z^2)^2}{d^2} - \frac{F}{d^2} \frac{\partial r}{\partial s} (1-Z^2)  \\
    & \quad + \frac{2F}{r} \sum_{i,j} \tilde \sigma^{ab} \tilde X^i_a \frac{ \partial \tilde y^i }{\partial s}X^j_b \tilde y^j +O(d^{-3}) \\
= &     4 +  \frac{\partial_s M(u)}{r} (1-  \sum_{i,j}\tilde X^i \tilde y^i  \tilde X^j \tilde y^j)   - \frac{1}{d^2}  \left[   FF'  Z(1-Z^2)^2 + 3 F Z(1-Z^2)  \right] +O(d^{-3})
\end{split}
\]
where \eqref{useful_expansion2} is used in the last equality. Moreover,
\[  
\begin{split}
\partial_s M(u) 
= &-\partial_u M(u) \frac{\partial r}{ \partial s} \\
=&  [F'+ \frac{1}{2d} F'' (1-Z^2)] [Z+ \frac{1}{d} (1-Z^2)] + O(d^{-2}).
\end{split}
\]
Collecting terms, we obtain
\begin{equation} \label{mean_curvature_term_3}
 \begin{split}
  &\frac{1}{2}  \sigma^{ab} \partial_s  \bar  g_{ab} \\
  = & 2 + \frac{1}{2d} F' Z(1-Z^2)\\
   &+ \frac{(1-Z^2)}{d^2} \left[  \frac{1}{4} F'' Z(1-Z^2)-  F'(2 Z^2- \frac{1}{2} )- \frac{3}{2}F Z - \frac{1}{2} FF'  Z (1-Z^2)\right ] +O(d^{-3}).
  \end{split}
 \end{equation}
 The lemma follows from combining \eqref{mean_curvature_term_1}, \eqref{mean_curvature_term_2} and \eqref{mean_curvature_term_3}.
\end{proof}
 \begin{corollary}
The condition \eqref{second_order_cond_2} holds for the family of surfaces $\Sigma_{d}$.
 \end{corollary}
 \begin{proof}
 From the proof of Theorem 5.2 in \cite{Chen-Wang-Yau5},
 \[ \int_{S^2} (h_0^{(-1)} -h^{(-1)}) \, {\rm d}V_{S^2}=- \int_{S^2} (V^{(-1)} + h^{(-1)}) \, {\rm d}V_{S^2}.\]
 As a result, we have
  \[ 
  \begin{split}
     \int_{S^2} (h_0^{(-1)} -h^{(-1)}) \, {\rm d}V_{S^2}
  =& -\frac{1}{2} \int_{S^2} F'Z(1-Z^2)+F(1-3Z^2) \, {\rm d}V_{S^2}\\
  =& -\frac{1}{2} \int_{S^2} \frac{\partial}{\partial Z} [FZ(1-Z^2)] \, {\rm d}V_{S^2}
  \end{split}
  \]
  which vanishes due to \eqref{vanish_by_part}.
 \end{proof}
We are now ready to compute $V^{(-2)} + \breve h^{(-2)}+V^{(-1)}  \breve h^{(-1)}$.
\[
\begin{split}
V^{(-2)} + \breve h^{(-2)}+V^{(-1)}  \breve h^{(-1)}
=&  \frac{1}{4} F'' Z(1-Z^2)^2 + \frac{1}{4}F'  (3 -14Z^2 + 11Z^4 )+ F ( 6Z^3 - 5 Z )\\
&+ \frac{1}{8} \Big [ F^2 (-1 + 10 Z^2   -11Z^4) - (F^2)'   (Z- 3Z^3 + 2 Z^5) \Big].\\\end{split}
\]
Moreover, we observe that 
\[  F^2 (-1 + 10 Z^2   -11Z^4) - (F^2)'   (Z- 3Z^3 + 2 Z^5)  =  -\frac{\partial}{\partial Z} [F^2(Z- 3Z^3 + 2 Z^5)] + Z^2 -Z^4\]
and 
\[
\begin{split}
    & \frac{1}{4} F'' Z(1-Z^2)^2 + \frac{1}{4}F'  (3 -14Z^2 + 11Z^4 )+ F ( 6Z^3 - 5 Z )\\
 = & \frac{1}{4}  \frac{\partial}{\partial Z}[F'Z(1-Z^2)^2 +F(2- 8Z^2 +6 Z^4)]   -F Z. \\ 
 \end{split}
\]
As a result, we conclude that 
\begin{equation}  \label{induced_energy_final} \int_{S^2}  \left [ V^{(-2)} + \breve h^{(-2)}+V^{(-1)}  \breve h^{(-1)}  \right ] {\rm d}V_{S^2}= \int_{S^2} \left [  -FZ + \frac{1}{8}F^2Z^2(1-Z^2)\right ] {\rm d}V_{S^2}. \end{equation}
Combining \eqref{optimal_energy_final} and \eqref{induced_energy_final}, we conclude that
\begin{equation}
E(\Sigma_{d},X,T_0) = \frac{1}{8 \pi d^2} \int_{S^2} Z F  {\rm d}V_{S^2} + O(d^{-3}).
\end{equation}
Applying \eqref{vanish_by_part} again, we have
\[  \int_{S^2} Z F  {\rm d}V_{S^2}  = \frac{1}{2}\int_{S^2} F' (1-Z^2){\rm d}V_{S^2}  .\]
We obtain the following theorem:
\begin{theorem}\label{main_theorem}
Let $X$ be an isometric embedding of $\Sigma_{d}$ into $\R^{3,1}$ such that  $X^0=O(d^{-1})$, and the pair $X$  and $T_0=(1,0,0,0)$ solves the leading order
term of the optimal embedding equation. The quasi-local energy is
\begin{equation}
E(\Sigma_{d},X,T_0) = -\frac{1}{16 \pi d^2} \int_{S^2}  \partial_uM (1-Z^2) {\rm d} V_{ S^2} + O(d^{-3}).
\end{equation}
\end{theorem}
In \cite{Chen-Wang-Yau5}, an invariant is defined on loops near the infinity in a gravitational perturbation of the Schwarzschild spacetime. We briefly recall the result here. See Section 8 of \cite{Chen-Wang-Yau5} for more details. Given a simple close curve $\gamma$ and a normal vector field $L$ along $\gamma$. For any smooth convex surface $\Sigma$ with boundary $\gamma$ that is tangent to $L$, let $h_{ab}$ be its second fundamental form and ${\rm d}V_{\Sigma}$ be its area form with respect to the Schwarzschild spacetime. Suppose the second fundamental form  and the area form admit the following expansions:
\[ 
\begin{split}
 h_{ab}=& h^{(0)}_{ab} + O(d^{-1}) \\
 {\rm d}V_{\Sigma} =& {\rm d}V^{(0)}_{\Sigma}+ O(d^{-1}).
 \end{split}
 \]
 Suppose we have a one-parameter family $g(z)$ of gravitational perturbations of the Schwarzschild spacetime where $g(0)$ is the Schwarzschild spacetime. Let $\sigma(z)$ and $H(z)$ be the induced metric and the mean curvature vector of  $\Sigma$ with respect to the metric $g(z)$, respectively. Consider the variations of the induced metric and the norm of the mean curvature vector
\[
\begin{split}
\delta \sigma = & \frac{{\rm d} \sigma(z)}{{\rm d} z}|_{z=0}\\
\delta H =& \frac{{\rm d}| H(z)|}{{\rm d} z}|_{z=0}
\end{split}
\]
where $\delta \sigma $ and  $\delta H $ admit the following expansions
\[
\begin{split}
\delta \sigma_{ab} = &  \frac{\delta \sigma^{(-1)}_{ab}}{d} + O(d^{-2})\\
\delta H = & \frac{\delta h^{(-1)}}{d} + O(d^{-2}).
\end{split}
\]
The invariant is defined as follows:
\begin{theorem} \label{Theorem 4.5}
Let $\gamma$ and $L$ as above. Given a one-parameter family of gravitational perturbations of the Schwarzschild spacetime satisfying the dominant energy condition, the following quantity
\[ |\frac{1}{8 \pi}\int_{\Sigma}  \left [ \frac{1}{2} (h^{(0)})^{ab} \delta\sigma_{ab}^{(-1)} + \delta h^{(-1)} \right]{\rm d}V^{(0)}_{\Sigma} |\]
is independent of the choice of $\Sigma$. 
\end{theorem}
We evaluate this invariant on loops in the Vaidya spacetime with respect to the Schwarzschild spacetime of mass $m>0$. Given the Vaidya spacetime,
\[ 
g(z)=-(1-  \frac{M(u)}{r}) du^2 -2dudr + r^2 (d\theta^2+\sin^2\theta d\phi^2),
\]
 we apply Theorem \ref{Theorem 4.5} to the following one-parameter family of spacetimes
\[ 
g(z)=-(1-  \frac{M(u)z + m(1-z)}{r}) du^2 -2dudr + r^2 (d\theta^2+\sin^2\theta d\phi^2)
\]
joining the Schwarzschild spacetime of mass $m$ and the given Vaidya spacetime. 

The construction of $\Sigma_{d}$ and the function $Z$ in Section 2 is independent of the mass aspect function. We apply the same construction in each $g(z)$. In particular, we have the surface $\Sigma_{d}$ in the Schwarzschild spacetime $g(0)$ and Theorem \ref{Theorem 4.5} is applicable. 

Given $-1 < c < 1$, let 
\[ \Sigma_{d, c} =  \Sigma_{d} \cap \{  Z \ge c \} \]
be the portion of $\Sigma_{d}$ with $Z \ge c$. In particular, $\gamma_c =\Sigma_{d} \cap \{  Z =c \} $ is the boundary of $\Sigma_{d,c}$. It is also easy to compute the normal vector field $L_c$ of $\gamma_c$ in $\Sigma_{d, c}$ with respect to the metric $g(0)$. Finally, let $S^2_c$ be the portion of $S^2$ with $Z \ge c$. $S^2_c$ is precisely the limit of $\Sigma_{d, c}$ as $d$ approaches infinity.
\begin{proposition}
Let $\gamma_c$ and $L_c$ be as above. For any smooth convex surface $\Sigma$ with boundary $\gamma_c$ that is tangent to $L_c$, we have
\[ |\frac{1}{8 \pi}\int_{\Sigma} \left [ \frac{1}{2} (h^{(0)})^{ab} \delta\sigma_{ab}^{(-1)} + \delta h^{(-1)} \right]{\rm d}V^{(0)}_{\Sigma}|   =  \frac{1}{8}  (1-c^2)| c (M(-c) - m)|. \]
\end{proposition}
\begin{proof}
By Theorem \ref{Theorem 4.5}, it suffices to evaluate the integral on $\Sigma_{d, c}$. On $\Sigma_{d, c}$, we have $h^{(0)}_{ab}  = \tilde \sigma_{ab}$ and ${\rm d}V^{(0)}_{\Sigma} = {\rm d} V_{ S^2_c} $. Hence, the invariant is
\[  |\frac{1}{8 \pi} \int_{S^2_c} \frac{1}{2} \tilde \sigma^{ab}\delta\sigma_{ab}^{(-1)} + \delta h^{(-1)} {\rm d} V_{ S^2_c}|. \]
From  \eqref{lemma 3.1}, \eqref{lemma 4.2} and \eqref{divergence_leading}, we have
\[  \frac{1}{2} \tilde \sigma^{ab}\delta\sigma_{ab}^{(-1)}= \frac{(F-m)(1-Z^2)}{2} \]
and 
\[  \delta  h^{(-1)} = -(F-m)(1-2Z^2) - \frac{1}{2}F' Z(1-Z^2).   \]
As a result, 
\[
\begin{split}
   \frac{1}{8 \pi} \int_{S^2_c}  \delta V^{(-1)} + \delta h^{(-1)} {\rm d} V_{ S^2_c} 
=&  -\frac{1}{16 \pi} \int_{S^2_c}   (F-m)(1-3Z^2)+ F' Z(1-Z^2)  {\rm d} V_{ S^2_c} \\
= &  -\frac{1}{ 8} \int_{Z=c}^1 \frac{\partial}{\partial Z} [(F-m)Z(1-Z^2) ] dZ \\
= & \frac{1}{8}  c (1-c^2) (M(-c)- m).
\end{split}
 \]
 \end{proof}
In particular, we observe that while the quasi-local mass on $\Sigma_{d}$ is related to the mass loss, the invariant on the loop is related to the mass aspect function.

\end{document}